\documentclass[12pt,reqno]{amsart}

\usepackage{amsmath, amsthm, amsopn, amssymb, enumerate}

\usepackage{latexsym}
\usepackage{graphics}
\usepackage{graphicx}
\usepackage{verbatim} 
\usepackage{hyperref}
\usepackage[mathscr]{eucal}
\usepackage[text={6.5in,8.2in},centering]{geometry}

\newtheorem{thm}{Theorem}[section]
\newtheorem{lemma}[thm]{Lemma}

\newtheorem{prop}[thm]{Proposition}

\newtheorem{example}[thm]{Example}

\newtheorem*{claim*}{Claim}

\theoremstyle{definition}

\newtheorem{defn}[thm]{Definition}

\def\A{\mathcal{A}}

\def\C{\mathcal{C}}

\def\K{\mathcal{K}}

\def\Q{\mathcal{Q}}

\def\U{\mathcal{U}}
\def\V{\mathcal{V}}

\def\Ex{\mathbb{E}}

\def\N{\mathbb{N}}
\def\Pr{\mathbb{P}}

\def\vx{\mathbf{x}}
\def\vy{\mathbf{y}}
\def\vz{\mathbf{z}}

\def\le{\leqslant}
\def\ge{\geqslant}

\def\eps{\varepsilon}
\def\<{\langle}
\def\>{\rangle}

\def\Bin{\textup{Bin}}

\date{\today}

\author{Gonzalo Fiz Pontiveros}

\author{Simon Griffiths}

\author{Robert Morris} 

\author{David Saxton}

\author{Jozef Skokan}

 \address{
   Gonzalo Fiz Pontiveros, Simon Griffiths, Robert Morris, David Saxton \hfill\break
    IMPA, Estrada Dona Castorina 110, Jardim Bot\^anico, Rio de Janeiro, RJ, Brasil
 }
 \email{\{gf232|sgriff|rob|saxton\}@impa.br}
 
  \address{
   Jozef Skokan \hfill\break
    Department of Mathematics, LSE, Houghton Street, London, WC2A 2AE, England, and 
    Department of Mathematics, University of Illinois, 1409 W. Green Street, Urbana, IL 61801
 }
 \email{j.skokan@lse.ac.uk}

\thanks{Research supported in part by: CNPq bolsas PDJ (GFP, SG, DS), a CNPq bolsa de Produtividade em Pesquisa (RM). This work was begun during a visit of JS to IMPA in November 2012}

\begin{document}

\title{The Ramsey number of the clique and the hypercube}

\begin{abstract}
The Ramsey number $r(K_s,Q_n)$ is the smallest positive integer $N$ such that every red-blue colouring of the edges of the complete graph $K_N$ on $N$ vertices contains either a red $n$-dimensional hypercube, or a blue clique on $s$ vertices. Answering a question of Burr and Erd\H{o}s from 1983, and improving on recent results of Conlon, Fox, Lee and Sudakov, and of the current authors, we show that $r(K_s,Q_n) = (s-1) \big( 2^n - 1 \big) + 1$ for every $s \in \N$ and every sufficiently large $n \in \N$.
\end{abstract}

\maketitle

\pagestyle{myheadings}
\markboth{}{}

\section{Introduction}
 
One of the the most extensively-studied problems in Combinatorics is that of determining the \emph{Ramsey numbers} of graphs, defined as follows. Given graphs $G$ and $H$, let $r(G,H)$ denote the minimum integer $N$ such that every red-blue colouring of $E(K_N)$ contains either a blue copy of~$G$ or a red copy of~$H$. It follows from the classical theorem of Ramsey~\cite{Ramsey} that $r(G,H)$ is finite for every pair of (finite) graphs, but its order of magnitude is known only in a few special cases. In this paper we answer a question of Burr and Erd\H{o}s~\cite{BE83} from 1983 by determining $r(K_s,Q_n)$, the Ramsey number of the clique on $s$ vertices and the hypercube on $2^n$ vertices, for every $s \in \N$ and all sufficiently large~$n$. 

The oldest and most famous examples of Ramsey numbers are those involving cliques. Indeed, it was proved by  Erd\H{o}s and Szekeres~\cite{ESz} and by Erd\H{o}s~\cite{E47} over sixty years ago that $2^{k/2} \le R(k) \le 4^k$, where we write $R(k,\ell) = r(K_k,K_\ell)$ and $R(k) = R(k,k)$. Despite extensive attempts (see, e.g.,~\cite{Sp75,T88,Conlon}) these bounds have been improved only very slightly. More progress has been made in the so-called `off-diagonal' case, where it was proved by Ajtai, Koml\'os and Szemer\'edi~\cite{AKSz} and Kim~\cite{Kim} that $R(3,k) = \Theta\big( k^2 / \log k \big)$ (see also~\cite{Sh83}, and the recent improvements in~\cite{BK2,FGM}). However, for every (fixed) $s \ge 4$, even the problem of determining the correct exponent of $k$ in $R(s,k)$ is wide open.

In this paper we shall study the Ramsey numbers of cliques versus sparse graphs, where the situation is somewhat simpler, and much more progress has been made. The systematic study of this problem was initiated in 1983 by Burr and Erd\H{o}s~\cite{BE83}, who conjectured that 
\begin{equation}\label{eq:def:sgood}
r(K_s,H) \, = \, (s - 1)\big( v(H) - 1 \big) + 1
\end{equation}
for every fixed $s$ and all `sufficiently sparse' connected graphs $H$; in particular, for all such $H$ with bounded average degree. The lower bound in~\eqref{eq:def:sgood} is easy to see: simply take a blue Tur\'an graph\footnote{More precisely, take a collection of $s-1$ disjoint red cliques, each of size $n-1$, and note that this colouring contains no connected red graph on $n$ vertices, and no blue graph $G$ of chromatic number $s$.}, and colour the remaining edges red. We remark that this conjecture was partly motivated by a result of Chv\'atal~\cite{Ch}, who had proved a few years earlier that~\eqref{eq:def:sgood} holds for every $s \in \N$ and every tree $H$ on $n$ vertices. 

The Burr-Erd\H{o}s Conjecture was disproved by Brandt~\cite{Br}, who showed that there exist bounded degree graphs~$H$ with $r(K_3,H) > c \cdot v(H)$, for $c$ arbitrarily large. However, the conjecture is known to hold for a large class of graphs with bounded average degree and poor expansion properties. More precisely, Burr and Erd\H{o}s~\cite{BE83} proved it for every graph $H$ of bounded bandwidth\footnote{The bandwidth of a graph $H$ is defined to be the minimum $\ell \in \N$ for which there exists an ordering $v_1,\ldots, v_n$ of the vertices of $H$ such that every edge $v_iv_j$ satisfies $|i-j| \le \ell$.} and Allen, Brightwell and Skokan~\cite{ABS} extended this result to graphs of bandwidth at most $o(\log n / \log\log n)$, and to bounded degree graphs of bandwidth~$o(n)$. Moreover, Nikiforov and Rousseau~\cite{NR} proved that~\eqref{eq:def:sgood} holds for every $O(1)$-degenerate graph which may be disconnected into components of size $o(n)$ by removing at most $n^{1-\eps}$ vertices; as was observed in~\cite{CFLS}, together with the `separator theorem' of Alon, Seymour and Thomas~\cite{AST} this implies that the Burr-Erd\H{o}s Conjecture holds for every sufficiently large graph $H$ which avoids a given minor, and hence for every large planar graph. We remark that in fact the main results of both~\cite{ABS} and~\cite{NR} are considerably more general than those stated above; we refer the reader to the original papers for the exact statements. 

Combining the results of~\cite{ABS} and~\cite{Br}, the situation for graphs of bounded degree is now reasonably well-understood: roughly speaking,~\eqref{eq:def:sgood} holds for those graphs which have poor expansion properties, and fails otherwise. Moreover, by the results of~\cite{NR}, it is known that~\eqref{eq:def:sgood} holds for a large class of graphs with bounded \emph{average} degree (but unbounded maximum degree). In particular, the main theorem of~\cite{NR} resolved all but one of the specific problems mentioned in~\cite{BE83}; in this paper, we resolve the final remaining question.

We shall study the $n$-dimensional hypercube, i.e., the graph with vertex set $\{0,1\}^n$ and edges between pairs of vertices which differ in exactly one coordinate. This important family of graphs appears naturally in many different contexts, and its properties have been extensively-studied, including those relating to Ramsey Theory. For example, it is a long-standing conjecture of Burr and Erd\H{o}s that $r(Q_n) = O(2^n)$, but the best known bounds (see~\cite{Con} and~\cite{FS}) are roughly the square of this function. We note also that $Q_n$ has average degree roughly $\log N$, and bandwidth roughly $N / \sqrt{\log N}$, where $N = 2^n$, and hence it is not covered by the results of either~\cite{ABS} or~\cite{NR}.  

The first important breakthrough in the study of the Ramsey numbers $r(K_s,Q_n)$ was obtained only recently, by Conlon, Fox, Lee and Sudakov~\cite{CFLS}, who gave an upper bound which is within a constant factor (depending on $s$) of the trivial lower bound. We remark that we shall use (a modified version of) their method in Section~\ref{CFLSsec}, below. In~\cite{FGMSS}, the ideas of~\cite{CFLS} were combined with a technique (`embedding in snakes') which utilized the low bandwidth of $Q_n$, to determine $r(K_3,Q_n)$ up to a factor of $1 + o(1)$ as $n \to \infty$. In this paper we shall use a quite different method to completely resolve the problem of Burr and Erd\H{o}s for all sufficiently large $n$. Our main theorem is as follows.

\begin{thm}\label{thm:main}
Let $s \in \N$. Then
\begin{equation}\label{eq:mainthm}
 r(K_s,Q_n) = (s-1)(2^n-1) + 1
\end{equation}
for every $n \ge n_0(s)$.
\end{thm}

\enlargethispage{\baselineskip}

We remark that the proof of Theorem~\ref{thm:main} moreover easily generalizes to the case in which the clique $K_s$ is replaced by an arbitrary graph $H$ of chromatic number $s$  (assuming that $n \ge n_0(H)$ is sufficiently large), see Section~\ref{ProofSec} for the details. 
The value of $n_0(s)$ given by our proof grows very quickly, like a tower function of height roughly $s^{2s}$. No doubt this bound is very far from best possible, and it would be interesting to determine the smallest function $n = n(s)$ such that~\eqref{eq:mainthm} holds; note that it fails badly for $n = 2$, since one can easily show that $r(K_s,C_4) \ge s^{3/2 + o(1)} \gg s$. We remind the reader that it is a famous open problem of Erd\H{o}s~\cite{E81} to prove that $r(K_s,C_4) \le s^{2-\eps}$ for all sufficiently large $s \in \N$.  

The strategy we shall use in order to prove Theorem~\ref{thm:main} is roughly as follows. Let~$G$ be a two-coloured complete graph with $(s-1)(2^n-1)+1$ vertices and no blue~$K_s$. We shall almost-partition\footnote{We write `almost-partition' to mean a partition of all but a $o(1)$-fraction of the vertices.} $V(G)$ into a large number of sets which are internally dense in red, and then attempt to join these by a large number of small, disjoint, red complete bipartite graphs. If we find a sufficiently large component of connections of this type, then we shall be able to embed $Q_n$ into $G_R$; otherwise we will be able to almost-partition the vertex set into $s-1$ red cliques. Using this `stability theorem', the result then follows easily.

To be a little more precise, let us consider the case $s = 3$. We first find an integer $a$, a collection of disjoint red cliques of size $2^n / \log \log a$, and a small set $X$, such that every vertex sends at most $2^{n-2^a}$ blue edges to those vertices of $G$ which are not covered by the red cliques, nor by~$X$. To do so, we simply decrease $a$ until we find a collection with the right properties, and note that the blue neighbourhood of any vertex is a red clique, since $G$ contains no blue triangle. We thus obtain an almost-partition into dense red sets. We next choose $m \sim (\log\log a)^{O(1)}$, and attempt to find disjoint red copies of $K_{t,t}$ between the dense red sets, where $t = {m \choose m/2}$. If we fail to do so, then by the well-known theorem of K\"ov\'ari, S\'os and Tur\'an~\cite{KST}, the blue graph between the two sets must be very dense. However, if we succeed sufficiently often, then we shall be able to find a copy of $Q_n$ in $G_R$. Indeed, using the density of the red sets, we may greedily complete our copies of $K_{t,t}$ to red copies of $Q_m$. We can then define an auxiliary two-colouring $H$, by treating each copy of $Q_m$ as a vertex, and colouring a pair blue if any of the matching edges between corresponding vertices is blue. Crucially, this colouring contains no blue clique on $R_{2^m}(3)$ vertices, where $R_r(s)$ denotes the $r$-colour Ramsey number, and is still sufficiently dense in red. By modifying the argument of~\cite{CFLS} (see also~\cite{FGMSS}, where we proved a similar proposition), we can efficiently embed $Q_{n-m}$ in $H_R$, which immediately gives us an embedding of $Q_n$ in $G_R$, as required.

The rest of the paper is organised as follows. In Section~\ref{SimonSec} we generalize the idea above in order to prove a general decomposition theorem for $K_s$-free graphs, which gives us our almost-partition of $V(G)$ into dense red sets. In Section~\ref{CFLSsec} we generalize a result from~\cite{FGMSS}, whose proof is based on the method introduced in~\cite{CFLS}, to show that we can find a red copy of $Q_n$ in a two-coloured complete graph $H$ on $\big( 1 + o(1) \big) 2^n$ vertices, as long as $H_B$ is reasonably sparse and $K_s$-free. It is very important here that $s$ is allowed to go to infinity with $n$, at a rate depending on the density of $H_B$, since we will apply the result to colourings with no blue clique on $R_{2^m}(s)$ vertices. In Section~\ref{MatchingSec} we prove a more general embedding lemma, by finding a maximal collection of disjoint red copies of $K_{t,t}$, as described above. This allows us to find either an embedding of $Q_n$ into a collection of dense red sets, or a dense blue $(s-1)$-partite graph which covers almost all of $V(G)$. Finally, in Section~\ref{ProofSec}, we put the pieces together and prove Theorem~\ref{thm:main}.

\subsection*{Notation} 

\enlargethispage{\baselineskip}

For convenience, we collect here some of the notation which we shall use throughout the paper. If~$G$ is a two-coloured complete graph, then we take the colours to be red and blue, and write $G_R$ and $G_B$ for the graphs formed by the red and blue edge sets respectively. We also write $N_B(u)$ and $d_B(u)$ for the neighbourhood and degree of a vertex in $G_B$, and $e_B(X,Y)$ and $d_B(X,Y) = e_B(X,Y) / |X||Y|$ for the number of blue edges with one endpoint in $X$ and the other in $Y$, and for the density of the graph $G_B[X,Y]$, and similarly for $G_R$.





For each $d \ge 0$ and $\vx=(x_1,\dots, x_d) \in \{0,1\}^d$, let 
\[
Q_{\vx} \, = \, \big\{ \vy = (y_1,y_2,\dots,y_n) \in \{0,1\}^n \,:\, y_i = x_i \mbox{ for each } 1 \le i \le d \big\}
\]
denote the subcube of $Q_n$ consisting of points whose \emph{initial} coordinates agree with $\vx$. We call this an \emph{initial subcube of co-dimension} $d$. Notice that if $Q_{\vx}$ has co-dimension $d$, then every vertex $v\in V(Q_{\vx})$ has exactly $d$ neighbours in $V(Q_n) \setminus V(Q_{\vx})$. We say that disjoint subcubes $Q$ and $Q'$ are \emph{adjacent} if there exist vertices $v\in V(Q)$ and $v'\in V(Q')$ which are adjacent in $Q_n$. Given two vectors $\vx \in \{0,1\}^d$ and $\vz \in \{0,1\}^{d'}$, we write $\vx \sim \vz$ if the subcubes $Q_{\vx}$ and $Q_{\vz}$ are adjacent, and note that~$Q_{\vx}$ and~$Q_{\vz}$ are adjacent if and only if~$\vx$ and~$\vz$ differ in precisely one coordinate in the first $\min\{d,d'\}$ coordinates. 


Throughout the paper, $\log$ denotes $\log_2$, and $\log_{(k)}$ denotes the $k^{th}$ iterated logarithm, so $\log_{(1)}(n)= \log(n)$ and $\log_{(k+1)}(n) = \log\big( \log_{(k)} (n) \big)$. We shall also omit ceiling and floor symbols, and trust that this will cause the reader no confusion.

\section{A structural description of $K_s$-free graphs}\label{SimonSec}

In this section we shall prove the following key proposition, which states that every $K_s$-free graph can be almost-partitioned into sparse sets. 

\begin{prop}\label{prop:simon}
Given $\eps > 0$ and $s\in \N$, there exists $K \in \N$ such that the following holds. Let $N \in \N$, and let $a(0), \ldots, a(K) \in \N$ be a sequence of integers with $a(0) = N$ and
\begin{equation}\label{eq:simonseqconditions}
\qquad a(i+1) \, \le \, \frac{\eps}{8} \cdot a(i) \qquad\mbox{for every } 0 \le i \le K-1.
\end{equation}
Then, for every $K_s$-free graph $G$ on $N$ vertices, there exists $i \in [K-1]$ and a family $\U$ of disjoint vertex sets such that:
\begin{itemize}
\item[$(a)$] $|\bigcup_{U\in \U}U| \ge (1-\eps)N$.\smallskip
\item[$(b)$] $|U| = a(i)$ for every $U \in \U$.\smallskip
\item[$(c)$] $\Delta(G[U]) \le a(i+1)$ for every $U\in \U$.
\end{itemize}
\end{prop}

We first give an overview of the proof and introduce a little more terminology. For convenience, let us fix throughout the rest of the section $\eps > 0$ and $s \in \N$, choose $K$ sufficiently large, and fix a $K_s$-free graph $G$ on $N$ vertices and integers $a(1), \ldots, a(K)$ satisfying~\eqref{eq:simonseqconditions}.

One natural approach to take, already alluded to in the Introduction, would be to choose a maximal collection of disjoint red cliques\footnote{Here we refer to edges of $G$ as `red' and non-edges as `blue'.} of size $b_1$, then a maximal collection of blue triangle-free sets of size $b_2$ in what remains, then a maximal collection of blue $K_4$-free sets of size $b_3$, and so on. These sets have the useful property that, inside each set $B_j$ of size $b_j$, all blue degrees are at most $b_{j-1}$. Indeed, since $B_j$ is $K_j$-free, every blue neighbourhood must be $K_{j-1}$-free, and so would have been chosen at the previous step if it was sufficiently large.

We shall take roughly this approach, with two important modifications. The first is that, since we shall in fact need the degrees inside each set to be much smaller than any of the other sets, we will enforce a `degree gap' by throwing away a small set of vertices at each step. The second is that, instead of choosing the $b_j$s to be a growing sequence (which seems more natural, since the restriction on their blue edges is becoming weaker), we shall choose them to be decreasing \emph{extremely} quickly. 

We now turn to some more formal preparation. The following definition will be useful.

\begin{defn}
Let $0 \le \beta \le 1$ and let $\emptyset \neq I \subseteq [K]$ be an interval. We say that a set $W \subseteq V(G)$ of vertices is {\em $(\beta,I)$-good} if, whenever $i,j \in I$ with $i < j$, there exists a family $\U = \U(i,j)$ of disjoint subsets of $W$ with the following properties:
\begin{itemize}
\item[$(a')$] $\big| \bigcup_{U \in \U} U \big| \ge \beta \cdot |W|$. \smallskip
\item[$(b')$] $|U| = a(i)$ for every $U \in \U$. \smallskip
\item[$(c')$] $\Delta\big( G[U] \big) \le a(j)$ for every $U \in \U$.
\end{itemize}
\end{defn}

Note that Proposition~\ref{prop:simon} is equivalent to the statement that $V(G)$ is $(1-\eps,I)$-good for some interval of the form $I = \{i, i+1\}$. In order to prove this, we shall define two sequences of sets, 
$$V(G) = W_1 \supseteq W_2 \supseteq \cdots \supseteq W_s \qquad \text{and} \qquad U_r \subseteq W_{r-1} \setminus W_r \quad \text{for } r = 2,\ldots,s,$$ 
and a sequence $0 = i_1 < \cdots < i_s \le K$ of integers such that, setting 
$$I_r = \big\{ i_r, \dots, i_{r-1} + k^{s - r + 2} \big\}$$
for each $r \in \{2,\ldots,s\}$ and some fixed $k \in \N$ defined below, we have:
\begin{itemize}
\item[$(i)$] $W_r$ contains no $K_r$-free set of size $a(i_r + k^{s-r+1})$ for every $r \in \{ 2,\dots, s \}$, \smallskip
\item[$(ii)$] $\big| \bigcup_{r=2}^{s} U_r \big| \ge \big( 1 - \eps/2 \big) N$, \smallskip
\item[$(iii)$] $U_r$ is $(1 - \eps/2,I_r)$-good for every $r \in \{ 2,\dots, s \}$, \smallskip
\item[$(iv)$] $I_2 \supseteq \cdots \supseteq I_s$ and $|I_s| \ge 2$.\smallskip
\end{itemize}
Having found such a pair of sequences, we easily obtain a family satisfying~$(a)$,~$(b)$ and~$(c)$ as follows. Indeed, by~$(ii)$-$(iv)$, there exists a cover of all but at most $\eps N$ vertices of $G$ by sets of size $a(i)$ and with maximum internal degrees at most $a(i+1)$, where $\{ i, i+1 \} \subseteq I_s$. The purpose of condition~$(i)$ is to allow us to construct the set $U_{r+1}$; in particular, note that any $K_{r+1}$-free subset of $W_r$ has maximum degree at most $a(i_r + k^{s-r+1})$.

We need one more important definition, which we shall use to choose the sets $U_r$. Given a set $W \subseteq V(G)$ and $a,r \in \N$, define
\[
 f\big( W, a, r \big) \, = \, \max\bigg\{ \Big| \bigcup_{U \in \U} U \Big| \,:\, \U \subseteq W^{(a)} \mbox{ is a family of disjoint $K_r$-free sets of size~$a$} \bigg\},
\]
and say that a family $\U \subset W^{(a)}$ of disjoint $K_r$-free $a$-sets is \emph{$(W,a,r)$-maximal} if 
$$\Big| \bigcup_{U \in \U} U \Big| \, = \, f\big( W,a,r \big).$$
We now turn to the formal proof of the proposition.

\begin{proof}[Proof of Proposition~\ref{prop:simon}]
The proposition is trivial when $s \le 2$ (since $e(G) = 0$), so let us assume that $s \ge 3$.  Set 
$$k \, =\, \left\lceil \frac{8s}{\eps} \right\rceil \qquad \text{and} \qquad K = k^{s},$$
and let $i_1 = 0$ and $W_1 = V(G)$. Trivially, $W_1$ contains no $K_1$-free set of size $a(K) \ge 1$. So let $r \ge 2$, and suppose that we have found an integer $i_{r-1} \in [K]$ and a set $W_{r-1} \subseteq V(G)$ which contains no $K_{r-1}$-free set of size $a(i_{r-1} + k^{s-r+2})$. We will show how to construct $i_r \in [K]$, $W_r \subseteq W_{r-1}$ and $U_r \subseteq W_{r-1} \setminus W_r$ with the required properties. 

Indeed, first define $i_r - i_{r-1}$ to be the least positive multiple of $k^{s-r+1}$ with the property that
\begin{equation}\label{eq:choosing:ir}
f\big( W_{r-1}, a(i_r +k^{s-r+1}), r \big) \, \le \, f\big( W_{r-1}, a(i_r), r \big) \, + \, \frac{\eps N}{4s}.
\end{equation}
In words, if increasing $i_r$ by $k^{s-r+1}$ would cause the size of the largest family of disjoint $K_r$-free $a(i_r)$-sets to increase by a positive fraction of $N$, then we do so; otherwise we stop. Note that 
\begin{equation}\label{eq:irbound}
0 \, \le \, i_r -  i_{r-1} \, \le \, \frac{4s}{\eps} \cdot k^{s-r+1} \, \le \,  \frac{1}{2} \cdot k^{s-r+2},
\end{equation}
since otherwise we would have increased $i_r$ so many times that $f\big( W_{r-1}, a(i_r), r \big) > N$, which is impossible. 

\enlargethispage{\baselineskip}

Now let $\U_r$ be an arbitrary $\big( W_{r-1}, a(i_r), r \big)$-maximal family, and set
$$U_r \, = \, \bigcup_{U \in \U_r} U \qquad \text{and} \qquad I_r \, =\, \big\{ i_r, \dots, i_{r-1} + k^{s - r + 2} \big\}.$$ 

\noindent \textbf{Claim:} $U_r$ is $\big( 1 - \eps / 2, I_r \big)$-good.

\begin{proof}[Proof of claim]
Let $i,j \in I_r$ with $i < j$. In order to obtain a refinement $\U_r(i,j)$ of $\U_r$ which satisfies~$(a')$~$(b')$ and~$(c')$, we simply $(1 - \eps / 2)$-cover each $U \in \U_r$ by disjoint sets of size $a(i)$ arbitrarily. To see that this is possible, simply note that either $i = i_r$, in which case $|U| = a(i)$, or $i > i_r$, in which case $a(i) \le \eps |U| / 8$ by~\eqref{eq:simonseqconditions}. To see that 
$$\Delta\big( G[X] \big) \, \le \, a\big( i_{r-1} + k^{s-r+2} \big) \, \le \, a(j)$$
for each $X \in \U_r(i,j)$, recall that each $Y \in \U_r$ is $K_r$-free, and that $W_{r-1} \supseteq Y \supseteq X$ contains no $K_{r-1}$-free set of size $a(i_{r-1} + k^{s-r+2})$. Thus all internal degrees of $X$ must be at most $a(i_{r-1}+k^{s-r+2})$, as claimed.
\end{proof}

To complete the induction step, we need to remove a small number of additional vertices from $W_{r-1}$, since we shall not be able to control their degrees in the next step. Set $W_r' = W_{r-1} \setminus U_r$ and let $X_r \subseteq W_r'$ be the union of a $\big( W_r', a(i_r + k^{s-r+1}), r \big)$-maximal family. We claim that 
\begin{equation}\label{eq:Xrbound}
|X_r| \,\le \, \frac{\eps N}{4s} + \frac{\eps |U_r|}{8}.
\end{equation}
Indeed, since $U_r$ can (as noted above) be $(1 - \eps / 8)$-covered by $K_r$-free sets of size $a(i_r + k^{s-r+1})$, it follows from~\eqref{eq:choosing:ir} that 
$$|X_r| + \left( 1 - \frac{\eps}{8} \right) |U_r| \, \le \, f\big( W_{r-1}, a(i_r + k^{s-r+1}), r \big) \, \le \, f\big( W_{r-1}, a(i_r), r \big) + \frac{\eps N}{4s} \, =\, |U_r| + \frac{\eps N}{4s}.$$ 
Set 
$$W_r \, = \, W_r' \setminus X_r \, = \, W_{r-1} \setminus \big( U_r \cup X_r \big),$$ 
and observe that $W_r$ contains no $K_r$-free set of size $a(i_r + k^{s-r+1})$, as required.

Finally, let us check that $U_2, \ldots, U_s$ and $I_2, \dots, I_s$ do indeed satisfy $(i)$-$(iv)$. We have already proved~$(i)$ and~$(iii)$ above. To see that $(ii)$ holds, note first that $|W_s| \le \eps N/8$, since $G$ is $K_s$-free and $a(i_s) \le \eps N / 8$, so $U_s$ fails to cover at most this many vertices of $W_{s-1}$. By~\eqref{eq:Xrbound}, it follows that
$$\bigg| V(G) \setminus \bigcup_{r=2}^{s}U_r \bigg| \, \le \, \frac{\eps N}{8} +\, \sum_{r=2}^{s} |X_r| \, \le \, \frac{3\eps N}{8} +\, \frac{\eps}{8}\sum_{r=2}^{s} |U_r| \, \le \, \frac{\eps N}{2}.$$
To verify $(iv)$, simply recall that $I_r$ was defined to be $\{i_r,\ldots,i_{r-1} + k^{s-r+2}\}$, and note that, by~\eqref{eq:irbound}, the sequence $i_r$ is increasing and the sequence $i_r + k^{s-r+1}$ is decreasing.  Finally, $I_s =\{i_s,\dots, i_{s-1}+k^2 \}$ has cardinality at least $k^2 - i_s + i_{s-1} \ge k^2/2 \ge 2$, as required.
\end{proof}

\section{An embedding lemma for dense red colourings}\label{CFLSsec}

In this section we shall adapt the method of Conlon, Fox, Lee and Sudakov~\cite{CFLS} to prove the following proposition. 

\begin{prop}\label{prop:dense_embed}
Given any $\eps > 0$ and $k \in \N$, there exists $n_0 = n_0(\eps,k)$ such that the following holds for every $n \ge n_0$ and every $s \le \log_{(k+1)} (n)$. If $H$ is a two-coloured complete graph on $(1 + \eps)2^n$ vertices with no blue~$K_s$ and
\begin{equation}\label{eq:densebluedegrees}
d_B(u) \, \le \, \frac{2^n}{\log_{(k)} (n)}
\end{equation}
for every $u \in V(H)$, then $Q_n \subset H_R$.
\end{prop}

We remark that the case $s = 3$ of Proposition~\ref{prop:dense_embed} was proved in~\cite{FGMSS}. As in that paper, our strategy will be roughly as follows: for each $r \in [k+1]$, we shall find an almost-partition of $V(H)$ into $2^{d(r)}$ dense red sets, for some $d(r)$ defined below, and an assignment of initial subcubes of $Q_n$ of co-dimenson $d(r)$ to these sets, such that there are few blue edges between pairs of sets which correspond to adjacent subcubes. We shall find such \emph{level-$r$ assignments} iteratively (see Lemma~\ref{lemma:refining_an_assignment}), refining the almost-partition at level $r$ to obtain that at level $r+1$. Once we have found a level-$(k+1)$ assignment, it will be straightforward (see Lemma~\ref{lemma:embed_into_assignment}) to embed $Q_n$ into $H_R$ greedily.

In order to refine an almost-partition $\A$ at level $r$, we assign a subset $C(\vy) \subset A(\vx) \in \A$ to each $\vy \in Q_{d(r+1)}$ in turn, where $\vx = \vy[d(r)]$ is the $d(r)$-initial segment of $\vy$. (Thus subcubes of a cube are assigned to subsets of the corresponding set.) We first remove from $A(\vx)$ the vertices to which a subcube has already been assigned, and the vertices which send too many blue edges to any set to which a neighbouring cube has already been assigned. By our degree conditions, we do not remove too many vertices in this process. Finally, we use the fact that $H_B$ is $K_s$-free to find a large dense red subset in what remains of $A(\vx)$ (see Lemma~\ref{lemma:finding_dense_subset}). 


We now turn to the details, which are straightforward but somewhat technical. Define
$$d(r) = \left\{
\begin{array}{cl}
\big( \log_{(k+2-r)} n \big)^3 & \mbox{ if } r \in \{1,\ldots,k+2\}, \smallskip \\
0 &\mbox{ otherwise,} 
\end{array}
\right.$$
and note that $d(r+1) \gg s \cdot d(r)$ and $2^{d(r+1)/s} \gg d(r+2)^{5k}$ for every $r \in \{0,\ldots,k\}$. As described above, after~$r$ stages we will have split~$Q_n$ into $2^{d(r)}$ subcubes, each of co-dimension~$d(r)$. Given vectors $\vx,\vx' \in Q_{d(r)}$, let
$$t(\vx,\vx') = \left\{
\begin{array}{cl}
\min \big\{ p \in [r] \,:\, \vx[d(p)] \ne \vx'[d(p)] \big\} & \mbox{ if } \vx \ne \vx', \smallskip \\
r+1 &\mbox{ otherwise,} 
\end{array}
\right.$$
where $\vx[\ell]$ denotes the $\ell$-initial segment of $\vx$. For convenience, let us set 
$$\gamma \, := \, \frac{\eps}{3(k+2)}$$ 
and fix (for the rest of this section) a two-coloured complete graph $H$ with no blue~$K_s$ and
\begin{equation}\label{eq:verticesH} 
v(H) \, = \, \Big( 1 + 3\big( k + 2 \big) \gamma \Big) 2^n 
\end{equation}
where $n \in \N$ is sufficiently large. Let us assume also that $H$ satisfies~\eqref{eq:densebluedegrees}. 

We are now ready for an important definition.

\begin{defn}
\label{defn:assignment}
Given $r \in \{0,\ldots,k+1\}$, a \emph{level-$r$ assignment} of $Q_n$ into $H$ is a collection $\A = \{A(\vx)\}_{\vx \in Q_{d(r)}}$ of disjoint sets of vertices of $H$, satisfying the following conditions:
\begin{enumerate}
 \item[$(a)$] \label{condn:ass_size} $|A(\vx)| =\Big( 1 + 3\big( k + 2 - r \big)\gamma \Big) 2^{n-d(r)}$ for every $\vx \in Q_{d(r)}$.\smallskip  
 \item[$(b)$] \label{condn:ass_degree} For every $\vx,\vx' \in Q_{d(r)}$ with $\vx = \vx'$ or $\vx\sim \vx'$, we have
 \[
  |N_B(v) \cap A(\vx)| \le \frac{2^{n-d(r)}}{d\big( t(\vx,\vx') \big)^{4(k+2-r)}} \qquad\mbox{for all } v \in A(\vx').
 \]
\end{enumerate}
\end{defn}

Note that a level-0 assignment always exists, since $d(0) = 0$ and we may set $A(\vx) = V(H)$ for the unique vertex $\vx \in Q_0$ (i.e., the vector of length zero). Condition~$(a)$ holds by~\eqref{eq:verticesH}, and Condition~$(b)$ holds by~\eqref{eq:densebluedegrees}, since $t(\vx,\vx) = 1$ and $\log_{(k)} n \gg \big( \log_{(k+1)} n \big)^{12(k+2)}$. 

Using this observation, Proposition~\ref{prop:dense_embed} follows easily from the next two lemmas. The first allows to inductively construct level-$r$ assignments for $r \in \{1,\ldots,k+1\}$.

\begin{lemma}\label{lemma:refining_an_assignment}
Let $r \in \{0,\ldots,k\}$. If $\A = \{A(\vx)\}_{\vx \in Q_{d(r)}}$ is a level-$r$ assignment of $Q_n$ into~$H$, then there exists a level-$(r+1)$ assignment $\C = \{C(\vy)\}_{\vy \in Q_{d(r+1)}}$ of $Q_n$ into $H$, with 
$$C(\vy) \subset A\big( \vy[d(r)] \big)$$ 
for every $\vy \in Q_{d(r+1)}$.
\end{lemma}

The second lemma allows us to embed $Q_n$ into $H_R$ via a level-$(k+1)$ assignment. 

\begin{lemma}
\label{lemma:embed_into_assignment}
If $\A = \{A(\vx)\}_{\vx \in Q_{d(k+1)}}$ is a level-$(k+1)$ assignment of $Q_n$ into $H$, then there exists an embedding $\varphi$ of~$Q_n$ into~$H_R$, such that $\varphi(Q_{\vx}) \subset A(\vx)$ for every $\vx \in Q_{d(k+1)}$.
\end{lemma}

Lemma~\ref{lemma:embed_into_assignment} follows from a straightforward greedy embedding. We shall prove it first, as a gentle warm-up for our main task, proving Lemma~\ref{lemma:refining_an_assignment}. 

\begin{proof}[Proof of Lemma~\ref{lemma:embed_into_assignment}]
Let $\vz_1\,\ldots,\vz_{2^n}$ be an arbitrary ordering of the vertices of $Q_n$. We claim that for each $j \in [2^n]$, there exists a vertex $v(j) \in A\big( \vz_j[d(k+1)] \big)$ such that 
$$v(j) \not\in \bigcup_{\vy \in T(j)} N_B\big( \varphi(\vy) \big),$$
where $T(j)$ denotes the collection of $Q_n$-neighbours of $\vz_j$ which have already been embedded.\footnote{That is, $T(j) = \{ \vz_i : i < j \textup{ and } \vz_i \sim \vz_j \}$.} If such a vertex exists for each $j$, then setting $\varphi(\vz_j) = v(j)$ gives the desired embedding. 

To see that such a vertex $v(j)$ exists, let~$\vy \sim \vz_j$ be a $Q_n$-neighbour of $\vz_j$ which has already been embedded as $\varphi(\vy)$. By Condition~$(b)$ in Definition~\ref{defn:assignment}, we have
\[
 \big| N_B\big( \varphi(\vy) \big) \cap A\big( \vz_j[d(k+1)] \big) \big| \, \le \, \frac{2^{n-d(k+1)}}{d\big( t(\vy[d(k+1)],\vz_j[d(k+1)] ) \big)^4}.
\]
Now, simply observe that $t\big( \vy[d(k+1)],\vz_j[d(k+1)] \big) \in [k+2]$ for every $\vy,\vz_j \in Q_n$, and that there are at most $d(r)$ vertices~$\vy \in Q_n$ with $\vy \sim \vz_j$ and $t\big( \vy[d(k+1)],\vz_j[d(k+1)] \big) = r$ for every $r \in [k+2]$, since $\vy$ and $\vz_j$ differ on exactly one of the first $d(r)$ coordinates. Hence
\begin{equation}\label{eq:embed:doublecounting}
 \sum_{\vy \in T(j)} \big| N_B\big( \varphi(\vy) \big) \cap A\big( \vz_j[d(k+1)] \big) \big| \, \le \, \sum_{r = 1}^{k+2} d(r) \cdot \frac{2^{n-d(k+1)}}{d(r)^4} \, \le \, \gamma \cdot 2^{n-d(k+1)},
\end{equation}
where the last inequality holds since $d(r) \gg 1$ as $n \to \infty$ for every $r \in [k+2]$. 

Finally, since there are at least 
$$\big| A\big( \vz_j[d(k+1)] \big) \big| - 2^{n-d(k+1)} \,>\, \gamma \cdot 2^{n-d(k+1)}$$ 
vertices in $A\big( \vz_j[d(k+1)] \big)$ which are not already in the image of~$\varphi$, there must exist a vertex $v(j)$ as claimed.
\end{proof}

We now turn to the proof of Lemma~\ref{lemma:refining_an_assignment}. We begin with a simple but key lemma, which follows from the fact that~$H_B$ is $K_s$-free. 

\begin{lemma}\label{lemma:finding_dense_subset}
For every $d \ge 0$ and $X \subset V(H)$, there exists a set $Y \subset X$ such that
 \begin{equation}\label{eq:Yconditions}
  |Y| \ge 2^{-(s-2)d}|X| \qquad \text{and} \qquad |N_B(v) \cap Y| \le 2^{-d}|Y|
 \end{equation}
for every $v \in Y$.
\end{lemma}

\begin{proof}
Let $0 \le i \le s-3$, and suppose that we have found a set $X_i \subset X$ such that
\begin{equation}\label{eq:findingYconditions}
H_B[X_i] \text{ is  $K_{s-i}$-free} \qquad \text{and} \qquad |X_i| \ge 2^{-id}|X|.
\end{equation}
When $i = 0$ this is clearly possible (simply set $X_0 = X$). We claim that either $Y = X_i$ satisfies~\eqref{eq:Yconditions}, or there exists a set $X_{i+1} \subset X_i$ satisfying~\eqref{eq:findingYconditions} for $i+1$. Indeed, if~\eqref{eq:Yconditions} does not hold with $Y = X_i$, then there exists $v \in X_i$ with $|N_B(v) \cap X_i| > 2^{-d}|X_i|$. But in that case we can set $X_{i+1} = N_B(v) \cap X_i$, which is $K_{s-i-1}$-free since $H_B[X_i]$ is $K_{s-i}$-free. 

Hence we obtain either a set $Y$ as required, or a set $X_{s-2}$ satisfying~\eqref{eq:findingYconditions} for $i = s - 2$. But in the latter case $H_B[X_{s-2}]$ is $K_2$-free, i.e., $X_{s-2}$ is a red clique, so we may set $Y=X_{s-2}$. 
\end{proof}


By choosing a random subset, we 
can moreover guarantee that the set $Y$ has whatever size we desire. We record this simple observation as the following lemma.

\begin{lemma}\label{lemma:random_subset}
Let $a,d \in \N$ and $X \subset V(H)$. If $\log|X| \cdot 2^{d+3} \le a \le 2^{-(s-2)d}|X|$, then there exists a set $Y \subset X$ such that
 \begin{equation}\label{eq:Yconditionsexact}
  |Y| = a \qquad \text{and} \qquad |N_B(v) \cap Y| \le 2^{-d+1}|Y|
 \end{equation}
for every $v \in Y$.
\end{lemma}



\begin{proof}
Let $Y'$ be the set obtained via Lemma~\ref{lemma:finding_dense_subset}, and let $Y$ be a uniformly chosen random subset of $Y'$ of size $a$. Set $p = a / |Y'|$ and let $B \sim \Bin\big( 2^{-d} |Y'|, p \big)$ denote the Binomial random variable. Using the inequalities of Pittel\footnote{Let $m,n \in \N$, and set $p = m/n$. For any property $\Q$ on $[n]$ we have $\Pr\big( \Q \text{ holds for a random $m$-set} \big) \le 3\sqrt{m} \cdot  \Pr\big( \Q \text{ holds for a random $p$-subset of $[n]$} \big)$.} and Chernoff (see, e.g.,~\cite{AS}), it follows that the expected number of vertices $v \in Y$ with more than $2^{-d+1}|Y|$ blue neighbours in $Y$ is at most 
$$|Y'|^{3/2} \cdot \Pr\Big( B > 2^{-d+1}a = 2 \cdot \Ex[B] \Big) \, \le \, |Y'|^{3/2} \cdot e^{-\Ex[B]/4} \, < \, 1,$$
since $\Ex[B] \ge 8 \log |Y'|$. Hence there must exist a set $Y$ with no such vertices, as required.
\end{proof}


We now show how to assign the subcubes of $Q_n$ to the dense red sets found in Lemma~\ref{lemma:finding_dense_subset}.

\begin{proof}[Proof of Lemma~\ref{lemma:refining_an_assignment}]
Let $r \in \{0,\ldots,k\}$, and let $\A = \{A(\vx)\}_{\vx \in Q_{d(r)}}$ be a level-$r$ assignment of $Q_n$ into $H$. Thus the $A(\vx)$ are disjoint sets of vertices of $H$, satisfying Conditions~$(a)$ and~$(b)$ of Definition~\ref{defn:assignment}. 
Our aim is to construct a level-$(r+1)$ assignment $\C = \{C(\vy)\}_{\vy \in Q_{d(r+1)}}$ of $Q_n$ into $H$, with 
$$C(\vy) \subset A\big( \vy[d(r)] \big)$$ 
for every $\vy \in Q_{d(r+1)}$.

Let $\vy_1,\ldots,\vy_{2^{d(r+1)}}$ be an arbitrary ordering of the vectors of $Q_{d(r+1)}$. Our main task will be to construct, one by one, disjoint sets
\[
 C(\vy_j)' \subset A\big( \vy_j[d(r)] \big)
\]
satisfying the following two conditions:
 \begin{enumerate}
\item[$(a')$] $|C(\vy_j)'| = \Big( 1 + \big( 3(k + 1 - r) + 1 \big)\gamma \Big) 2^{n-d(r+1)}$ for every $1 \le j \le 2^{d(r+1)}$.\smallskip
\item[$(b')$] For every $i \le j$ with $\vy_i = \vy_j$ or $\vy_i \sim \vy_j$, and every $v \in C(\vy_j)'$,
$$\big| N_B(v) \cap C(\vy_i)' \big| \, \le \, \frac{2^{n-d(r+1)}}{d\big( t(\vy_i,\vy_j) \big)^{4(k+1-r)+2}}.$$
\end{enumerate}
Once sets $C(\vy_j)'$ satisfying~$(a')$ and~$(b')$ have been found, it will be straightforward to find slightly smaller sets $C(\vy_j) \subset C(\vy_j)'$ satisfying the slightly stronger Condition~$(b)$. 

\smallskip
In order to find sets $C(\vy_j)'$ as described above, let $j \ge 1$ and suppose that we have already found sets $C(\vy_1)',\ldots,C(\vy_{j-1})'$ satisfying~$(a')$ and~$(b')$. We shall find $C(\vy_j)'$ inside $U(\vy_j) \subset A\big(\vy_j[d(r)] \big)$, where
\[
 U(\vy_j) := A\big(\vy_j[d(r)] \big) \setminus \bigcup_{i<j} C(\vy_i)',
\]
denotes the set of not yet occupied vertices in $A\big(\vy_j[d(r)] \big)$. We claim that 
\begin{equation}\label{eq:Ubound}
|U(\vy_j)| \ge 2\gamma \cdot 2^{n-d(r)}.
\end{equation} 
To see this simply note that 
$$\big| \big\{ i < j : C(\vy_i)' \subset A\big(\vy_j[d(r)] \big) \big\} \big| \, = \, \big| \big\{ i < j : \vy_i[d(r)] = \vy_j[d(r)] \big\} \big| \, < \, 2^{d(r+1)-d(r)},$$
and that
\[
\Big( 1 + 3(k+2-r)\gamma \Big) 2^{n-d(r)} \,-\, 2^{d(r+1)-d(r)} \Big( 1 + \big( 3(k+1-r)+1 \big) \gamma \Big) 2^{n-d(r+1)} \, = \, 2\gamma \cdot 2^{n-d(r)}.
\]
The bound~\eqref{eq:Ubound} now follows immediately. 

In order to find $C(\vy_j)' \subset U(\vy_j)$ satisfying condition~$(b')$, we must remove from~$U(\vy_j)$ all vertices which have high blue degree to some already-chosen set~$C(\vy_i)'$ with $\vy_i \sim \vy_j$. Let 
$$T(j) \, = \, \Big\{ 1 \le i \le j-1 \,:\, \vy_i \sim \vy_j \Big\}$$ 
denote the collection of indices of already-assigned $Q_{d(r+1)}$-neighbours of~$\vy_j$, and for each $i \in T(j)$, let
\begin{equation}\label{def:D}
 D_i(\vy_j) \, = \, \bigg\{ v \in U(\vy_j) \,:\, |N_B(v) \cap C(\vy_i)'| \ge \frac{2^{n-d(r+1)}}{d\big( t(\vy_i,\vy_j) \big)^{4(k+1-r)+2}} \bigg\}.
\end{equation}
We claim that
\begin{equation}\label{eq:Dbound}
 |D_i(\vy_j)| \, \le \, \frac{2^{d(r+1)-d(r)}}{d\big( t(\vy_i,\vy_j) \big)^2} \cdot \big| C(\vy_i)' \big| \, \le \, \frac{2^{n-d(r)+1}}{d\big( t(\vy_i,\vy_j) \big)^2}.
\end{equation}
Indeed, since $\vy_i \sim \vy_j$, it follows that either $\vy_i[d(r)]  = \vy_j[d(r)]$ or $\vy_i[d(r)] \sim \vy_j[d(r)]$, and hence Condition~$(b)$ in Definition~\ref{defn:assignment} holds, with $\vx = \vy_j[d(r)]$, for every $v \in C(\vy_i)'$. Thus, by the definition~\eqref{def:D} of $D_i(\vy_j)$ and the fact that $D_i(\vy_j) \subset A\big( \vy_j[d(r)] \big)$, we have
\begin{align*}
& \big| D_i(\vy_j) \big| \cdot \frac{2^{n-d(r+1)}}{d\big( t(\vy_i,\vy_j) \big)^{4(k+1-r)+2}} \, \le \, e_B\big( C(\vy_i)', D_i(\vy_j) \big) \\
& \hspace{2.5cm} \le \, e_B\big( C(\vy_i)', A(\vy_j[d(r)]) \big) \, \le \, \big| C(\vy_i)' \big| \cdot \frac{2^{n-d(r)}}{d\big( t(\vy_i,\vy_j) \big)^{4(k+1-r)+4}}.
\end{align*}
The first bound in~\eqref{eq:Dbound} follows immediately; the second follows by Condition~$(a')$. 

Now, set
$$X \, = \, U(\vy_j) \setminus \bigcup_{i\in T(j)} D_i(\vy_j).$$
We claim that 
\begin{equation}\label{eq:Xproperties}
|X| \ge \gamma \cdot 2^{n-d(r)} \qquad \text{and} \qquad \big| N_B(v) \cap C(\vy_i)' \big| \, \le \, \frac{2^{n-d(r+1)}}{d\big( t(\vy_i,\vy_j) \big)^{4(k+1-r)+2}}
\end{equation}
for every $i < j$ such that $\vy_i \sim \vy_j$ and every $v \in X$. To prove the bound on $|X|$, recall from~\eqref{eq:Ubound} that $|U(\vy_j)| \ge 2\gamma \cdot 2^{n-d(r)}$, and observe (cf.~\eqref{eq:embed:doublecounting}) that, by~\eqref{eq:Dbound}, we have
\[
 \sum_{i \in T(j)} \big| D_i(\vy_j) \big| \, \le \, \sum_{p=1}^{r+1} d(p) \cdot \frac{2^{n-d(r)+1}}{d(p)^2} \, \le \, \gamma \cdot 2^{n-d(r)}.
\]
Indeed, note that $t(\vy_i,\vy_j) \in [r+1]$ for every~$i \in T(j)$, and that there are at most $d(p)$ elements~$i\in T(j)$ with $\vy_i\sim \vy_j$ and $t(\vy_i,\vy_j) = p$, for every $p \in [r+1]$. Once again, the final inequality holds since~$n$ is sufficiently large and $d(p) \to \infty$ as $n \to \infty$ for every $p \in [r+1]$. The second inequality in~\eqref{eq:Xproperties} follows directly from the definitions of $X$, $D_i(\vy_j)$ and $T(j)$. 

Finally, we apply Lemma~\ref{lemma:random_subset} to $X$, with $a = 2^{n-d(r+1)+1}$ and $d = d(r+1)/s$, to obtain a set $Y \subset X$ with
\begin{equation}\label{eq:Yproperties}
|Y| \, = \, 2^{n-d(r+1)+1} \qquad \text{and} \qquad \big| N_B(v) \cap Y \big| \, \le \, \frac{2^{n-d(r+1)}}{d(t(\vy_j,\vy_j))^{4(k+1-r)+2}}
\end{equation}
for every $v \in Y$. To see that these values of $a$ and $d$ are compatible, simply note that 
$$|Y| \, = \, 2^{n-d(r+1) + 1} \, \ll \, \gamma \cdot 2^{n-d(r)-(s-2)d(r+1)/s} \, \le \, 2^{-(s-2)d}|X|,$$
since $d(r+1) / s \gg d(r)$. To obtain the bound on $|N_B(v) \cap Y|$, observe that
$$\big| N_B(v) \cap Y \big| \, \le \, 2^{-d+1} |Y| \, \ll \, \frac{2^{n-d(r+1)}}{d(r+2)^{4(k+1-r)+2}} \, = \, \frac{2^{n-d(r+1)}}{d\big( t(\vy_j,\vy_j) \big)^{4(k+1-r)+2}}$$
since $2^{d(r+1)/s} \gg d(r+2)^{5k}$. We may now choose $C(\vy_j)'$ to be an arbitrary subset of $Y$ of the correct size, since every such set automatically satisfies conditions~$(a')$ and~$(b')$.\smallskip

We have now constructed sets $C(\vy_j)'$ satisfying conditions~$(a')$ and~$(b')$; it only remains to find subsets $C(\vy_j) \subset C(\vy_j)'$ satisfying conditions~$(a)$ and~$(b)$. For each $i<j$ with $\vy_i \sim \vy_j$, condition~$(b')$ implies that every vertex in $C(\vy_j)'$ sends few blue edges to $C(\vy_i)'$. However there may be vertices in $C(\vy_i)'$ that send many blue edges to $C(\vy_j)'$, which we must remove. The proof that there are few such vertices is an easy exercise in double-counting, exactly as above. To spell it out, let 
$$\hat{T}(j) \, = \, \Big\{ j+1 \le i \le 2^{d(r+1)} \,:\, \vy_i \sim \vy_j \Big\}$$ 
denote the collection of indices of $Q_{d(r+1)}$-neighbours of~$\vy_j$ which were assigned later in the process than~$\vy_j$, and for each $i \in \hat{T}(j)$, let
\begin{equation}\label{def:Dhat}
 \hat{D}_i(\vy_j) \, = \, \bigg\{ v \in C(\vy_j)' \,:\, |N_B(v) \cap C(\vy_i)'| \ge \frac{2^{n-d(r+1)}}{d\big( t(\vy_i,\vy_j) \big)^{4(k+1-r)}} \bigg\}.
\end{equation}
We claim that
\begin{equation}\label{eq:Dhatbound}
 |\hat{D}_i(\vy_j)| \, \le \, \frac{\big| C(\vy_i)' \big|}{d\big( t(\vy_i,\vy_j) \big)^2}  \, \le \, \frac{2^{n-d(r+1)+1}}{d\big( t(\vy_i,\vy_j) \big)^2}.
\end{equation}
Indeed, by~\eqref{def:Dhat}, Condition~$(b')$, and the fact that $\hat{D}_i(\vy_j) \subset C(\vy_j)'$, we have
\begin{align*}
& \big| \hat{D}_i(\vy_j) \big| \cdot \frac{2^{n-d(r+1)}}{d\big( t(\vy_i,\vy_j) \big)^{4(k+1-r)}} \, \le \, e_B\big( C(\vy_i)', \hat{D}_i(\vy_j) \big) \\
& \hspace{2.5cm} \le \, e_B\big( C(\vy_i)', C(\vy_j)' \big) \, \le \, \big| C(\vy_i)' \big| \cdot \frac{2^{n-d(r+1)}}{d\big( t(\vy_i,\vy_j) \big)^{4(k+1-r)+2}}.
\end{align*}
as claimed. Finally, note that 
\[
 \sum_{i \in \hat{T}(j)} \big| \hat{D}_i(\vy_j) \big| \, \le \, \sum_{p=1}^{r+1} d(p) \cdot \frac{2^{n-d(r+1)+1}}{d(p)^2} \, \le \, \gamma \cdot 2^{n-d(r+1)}, 
\]
since there are at most $d(p)$ elements~$i\in \hat{T}(j)$ with $\vy_i\sim \vy_j$ and $t(\vy_i,\vy_j) = p$. 

Using Condition~$(a')$, it follows that
$$\bigg| C(\vy_j)' \setminus \bigcup_{i \in \hat{T}(j)} \hat{D}_i(\vy_j) \bigg| \, \ge \, \Big( 1 + 3\big( k+1-r \big) \gamma \Big)2^{n-d(r+1)}.$$ 
Hence, taking~$C(\vy_j)$ to be any subset of the correct size, and repeating the above for each $1 \le j \le 2^{d(r+1)}$ in turn,  we obtain a level-$(r+1)$ assignment, as required.
\end{proof}

We can now easily deduce Proposition~\ref{prop:dense_embed} from Lemmas~\ref{lemma:refining_an_assignment} and~\ref{lemma:embed_into_assignment}.

\begin{proof}[Proof of Proposition~\ref{prop:dense_embed}]
As noted earlier, our assumptions on~$H$ trivially imply that $\{V(H)\}$ is a level-$0$ assignment of~$Q_n$ into~$H$, provided~$n$ is sufficiently large. Now, applying Lemma~\ref{lemma:refining_an_assignment} for each $r \in \{0,\ldots,k\}$ in turn, we obtain a level-$(k+1)$ assignment of~$Q_n$ into~$H$. Finally, by Lemma~\ref{lemma:embed_into_assignment}, it follows that $Q_n \subset H_R$, as required.
\end{proof}

\section{Embedding with matchings}\label{MatchingSec}

In the previous section, we showed how to efficiently embed the hypercube $Q_n$ in a dense red colouring $H$ which is blue $K_s$-free. In this section we will strengthen Proposition~\ref{prop:dense_embed} in the following way. We shall show that if $V(H)$ may be partitioned into a reasonably small family of sets such that $H_R$ is dense on each, then either we can still efficiently embed the hypercube, or $H$ contains a large and dense blue $r$-partite graph. 

More precisely, we shall prove the following proposition. 

\begin{prop}\label{prop:matching_embed}
Given any $C,\eps > 0$ and $s,k \in \N$, there exists $n_0 = n_0(C,\eps,s,k) \in \N$ such that the following holds whenever $n \ge n_0$. Set $p = \log_{(k+4)} (n)$ and $q = \log_{(k)} (n)$.

Let $G$ be a two-coloured complete graph on at most $C \cdot 2^n$ vertices with no blue~$K_s$, and let $\U$ be a collection of disjoint sets of vertices such that:
\begin{itemize}
 \item[$(a)$] $|U| \ge 2^{n-p}$ for every $U \in \U$.\smallskip
 \item[$(b)$] $|\Delta_B(U)| \le 2^{n-q}$ for every $U \in \U$.
\end{itemize}
Then either there exists a partition $\U = \U_1 \cup \cdots \cup \U_r$ and a set $X \subset V(G)$ such that 
\begin{itemize}
 \item[$(i)$] $|X| \le \eps 2^n$ and $\sum_{U \in \U_i} |U| \le  \big( 1 + \eps \big) 2^n$ for every $i \in [r]$, and\smallskip
 \item[$(ii)$] $d_B\big( U_i \setminus X,U_j \setminus X \big) \le 1/n^4$ for every $U_i \in \U_i$ and $U_j \in \U_j$ with $i,j \in [r]$ and $i \neq j$,
\end{itemize}
or $Q_n \subset G_R$.
\end{prop}


In order to motivate the proof of Proposition~\ref{prop:matching_embed}, let us begin by discussing the following special case. Let $|\U| = 2$, and suppose moreover that $U_1$ and $U_2$ each have size $(1+\eps)2^{n-1}$. Suppose that there exists a perfect red matching $M$ between $U_1$ and $U_2$, and consider the two-colouring of $E(K_{|M|})$ obtained by identifying the endpoints of $M$ and taking the union of the blue graphs. Since $G$ contains no blue $K_s$, this colouring contains no blue $K_{R(s)}$, where $R(s)$ denotes the Ramsey number of $s$. Hence, by Proposition~\ref{prop:dense_embed}, it contains a red~$Q_{n-1}$, which (using $M$ to join the two sides) corresponds to a red copy of~$Q_n$ in the original colouring. We thus obtain the following result.

\begin{example}\label{example:matching_embed}
Let $G$ be a two-coloured complete graph with no blue~$K_s$, as in Proposition~\ref{prop:matching_embed}, and suppose that $\U = \{U_1,U_2\}$, and that $|U_1|,|U_2| \ge (1+\eps)2^{n-1}$. If~$G$ contains a perfect matching of red edges between~$U_1$ and~$U_2$, then $Q_n \subset G_R$.
\end{example}

In order to prove Proposition~\ref{prop:matching_embed}, we shall generalize this idea by decomposing $Q_n$ as $Q_m \times Q_{n-m}$ (instead of $Q_1 \times Q_{n-1}$), and by allowing the copies of $Q_m$ to span several dense red sets, rather than just two. In order to find such copies of~$Q_m$, it will be sufficient to find many red copies of $K_{t,t}$, where $t \ge \binom{m}{m/2}$.\footnote{We remind the reader of the well-known fact that a graph which contains no copy of $K_{t,t}$ must necessarily be rather sparse, see Theorem~\ref{thm:KST} below.} We next give another example, again in the case $|\U| = 2$, to illustrate this crucial observation.

\begin{example}\label{example:matching2}
Let $G$ be as in Example~\ref{example:matching_embed}. If~$G_R[U_1,U_2]$ contains at least $(1+\eps)2^{n-m}$ disjoint copies of $K_{t,t}$, where $t = \binom{m}{m/2}$ and $m = o\big( \log \log q \big)$, then $Q_n \subset G_R$.
\end{example}

\begin{proof}
We first attempt to greedily extend each copy of $K_{t,t}$, one by one, to a (disjoint) red copy of $Q_m$ with one half in $U_1$ and the other in $U_2$. By assumption~$(b)$, at most $m 2^{n-q} \ll 2^n$ vertices are forbidden at each step, and so we will succeed in covering all but $o(2^n)$ vertices of $U_1 \cup U_2$. Next we define a two-colouring $H$ of the complete graph on $M = (1 + \eps) 2^{n-m}$ vertices by identifying all points in the same copy of $Q_m$, and colouring an edge blue if any of the corresponding $2^m$ matching edges\footnote{That is, the edges required between the two copies of $Q_m$ to create a copy of $Q_{m+1}$.}  were blue in $G$. Note that every vertex has blue degree at most $2^{n-q+m}$ in $H$, and moreover it contains no blue clique on $R_{2^m}(s)$ vertices, where $R_r(s)$ denotes the $r$-colour Ramsey number of $s$. Since $R_r(s) \le r^{rs}$ for every $r,s \ge 2$, it follows by Proposition~\ref{prop:dense_embed} and our choice of $m$ that $H$ contains a red copy of $Q_{n-m}$, which corresponds to a red copy of $Q_n$ in $G$, as required.
 \end{proof} 

In order to prove Proposition~\ref{prop:matching_embed}, we shall need to perform a similar embedding into a large string of sets, not necessarily all the same size, which are connected by many vertex-disjoint copies of $K_{t,t}$. The following definition will be useful.

\begin{defn}
Let $\V$ be a collection of disjoint sets of vertices of a two-coloured complete graph $G$, let $m \in \N$ and set $t = \binom{m}{m/2}$. We say that $\V$ is an \emph{$m$-path in $G_R$} if there exists an ordering $(V_1,\ldots,V_\ell)$ of the members of $\V$ such that, for every $1 \le i < \ell$, the graph $G_R[V_i,V_{i+1}]$ contains $(1 + \eps) 2^{n-m}$ vertex-disjoint copies of $K_{t,t}$.
\end{defn}

The following embedding lemma is the key step in the proof of Proposition~\ref{prop:matching_embed}.

\begin{lemma}\label{lem:path:embed}
Given any $\eps > 0$ and $s,k \in \N$, there exists $n_0 = n_0(\eps,s,k) \in \N$ such that the following holds whenever $n \ge n_0$. Let $p = \log_{(k+4)}(n)$, $q = \log_{(k)}(n)$ and $2^{6p} \ll m \ll 2^{2^p}$. 

Let $G$ be a two-coloured complete graph with no blue~$K_s$, and let~$\V$ be a collection of disjoint sets of vertices of $G$ satisfying the following conditions:
\begin{itemize}
 \item[$(a)$] $\sum_{V \in \V} |V| \ge (1+3\eps)2^n$. \smallskip 
 \item[$(b)$] $|V| \ge 2^{n-3p}$ for every $V \in \V$.\smallskip  
 \item[$(c)$] $|\Delta_B(V)| \le 2^{n-q}$ for every $V \in \V$. \smallskip
 \item[$(d)$] $\V$ is an $m$-path in $G_R$.
\end{itemize}
Then $Q_n \subset G_R$.
\end{lemma}

Given $m \in \N$ and $0 \le a \le b \le m$, let us write
\[
 Q_m[a,b] \, = \, \Big\{ \vx \in Q_m : a \le |\vx| \le b \Big\}
 \]
for the subgraph of the hypercube $Q_m$ induced by the layers corresponding to sets of size between $a$ and $b$. We shall embed into each set $V \in \V$ a copy of $Q_m[a,b] \times Q_{n-m}$, where $a$ and $b$ are chosen so that $v\big(Q_m[a,b] \big)$ is proportional to $|V|$.  With this notation in hand, we are ready to prove the embedding lemma.

\begin{proof}[Proof of Lemma~\ref{lem:path:embed}] 
Let $\eps > 0$ and $s,k \in \N$ be arbitrary, let $m,n,p,q \in \N$ be as described in the statement of the lemma, and let $G$ be a two-coloured complete graph with no blue~$K_s$. Let~$\V$ be a collection of disjoint vertex sets satisfying conditions $(a)$, $(b)$, $(c)$ and $(d)$ of the lemma, and let $(V_1,\ldots,V_\ell)$ be an ordering of the elements of $\V$ such that $G_R[V_i,V_{i+1}]$ contains at least $M = (1 + \eps) 2^{n-m}$ vertex-disjoint copies of $K_{t,t}$ for every $1 \le i < \ell$.  

The first step is to choose an $\ell' \le \ell$ and a sequence of integers
$$0 \le b(1) < \dots < b(\ell') = m$$ 
such that, if we set $a(1) = 0$ and $a(i) = b(i-1) + 1$ for each $2 \le i \le \ell'$, then
\begin{equation}\label{eq:choosingells}
|V_i| \, \ge \, \big( 1 + 2\eps \big) 2^{n-m} \cdot v\big(Q_m[a(i),b(i)] \big)
\end{equation}
for every $i \in [\ell']$. We do so greedily, by choosing $b(i)$ to be maximal such that~\eqref{eq:choosingells} holds. To see that this works, simply note that,  
$$\sum_{i \in [\ell]} |V_i| \, \ge \, \big( 1 + 3\eps \big) 2^n  \, \ge \, \big( 1 + 2\eps \big) 2^{n-m} \bigg( 2^m + \ell' \cdot {m \choose m/2} \bigg),$$
by property~$(a)$. The second inequality above holds since $\ell' \ll \sqrt{m}$, which follows from property~$(b)$ and the fact that $2^p \ll \sqrt{m}$ and $v\big(Q_m[a,a] \big) \le t = O(2^m/\sqrt{m})$. 

Next, for each $1 \le i < \ell'$, choose from $G[V_i,V_{i+1}]$ a collection of $M$ disjoint copies of the complete bipartite graph with part sizes ${m \choose b(i)}$ and ${m \choose a(i+1)}$, and note that we may do so by our assumption, together with the fact that ${m \choose j} \le t$ for every $j \in [m]$. Let us write  
$$S^{(1)}_i,\ldots,S^{(M)}_i \subset V_{i} \qquad \text{and} \qquad T^{(1)}_i,\ldots,T^{(M)}_i \subset V_{i+1}$$ 
for the vertex sets of these complete bipartite graphs. 

We now greedily extend, for each $j \in [M]$, the graph induced (in $G_R$) by $\bigcup_{i = 1}^{\ell'} S^{(j)}_i \cup T^{(j)}_i$ to a red copy of $Q_m$, which we shall call $Q_m^{(j)}$. To do so, simply note that by condition~$(b)$, every vertex sends at most $2^{n-q}$ blue edges into its own part, and hence we will run out of space only when all but $m 2^{n-q} = o(|V|)$ vertices of some set $V \in \V$ have already been used. By~\eqref{eq:choosingells}, this will not happen before we have completed all $M$ copies of $Q_m$. 

Finally, consider the two-colouring $H$ of $K_M$ obtained by identifying the vertices of each copy of $Q_m$, and placing a blue edge between two vertices if any of the corresponding $2^m$ matching edges of $G$ were blue. To be precise, set $V(H) = \big\{ Q_m^{(1)}, \ldots, Q_m^{(M)} \big\}$ and 
$$E( H_R ) \, = \, \Big\{ \big\{ Q_m^{(i)}, Q_m^{(j)} \big\} \,:\, \vx^{(i)} \vx^{(j)} \in E(G_R) \text{ for every } \vx \in Q_m \Big\},$$ 
where $\vx^{(i)}$ denotes the vertex of $Q_m^{(i)}$ corresponding to $\vx$. The crucial observation is that $H$ contains no blue clique on $R_{2^m}(s)$ vertices, 
since $G_B$ is $K_s$-free. Indeed, we may colour each blue edge of $H$ by an arbitrary element $\vx \in Q_m$ such that $\vx^{(i)} \vx^{(j)} \in E(G_B)$; it is easy to see that a monochromatic $s$-clique in this colouring corresponds to a blue copy of $K_s$ in $G$. 

It only remains to check that $H$ satisfies the conditions of Proposition~\ref{prop:dense_embed}. To see this, recall that $R_r(s) \le r^{rs}$ (this follows by a simple induction), and observe every vertex in $H$ has blue degree at most $2^{n-q+m}$ in $H$. Moreover, since $m = o(\log \log q)$, it follows that $R_{2^m}(s) \ll q$. Hence, by Proposition~\ref{prop:dense_embed}, the two-colouring $H$ contains a red copy of $Q_{n-m}$, which corresponds to a red copy of $Q_n$ in $G$, as required.
\end{proof}

We shall next use Lemma~\ref{lem:path:embed} to find a red copy of $Q_n$ in a slightly more general structure. Note that the disjoint copies of $K_{t,t}$ in the proof above only covered a small proportion of the vertices of each set $V \in \V$. This motivates the following definition.

\begin{defn} \label{defn:m_good}
Let $G$ be a two-coloured complete graph. For each $m \in \N$ and $\gamma > 0$, we say that a pair $\{U_1,U_2\}$ of disjoint sets of vertices of $G$ is {\em $(m,\gamma)$-good} if 
$$G_R[X_1,X_2] \textup{ contains a copy of $K_{t,t}$, where $t = \binom{m}{m/2}$}$$
for every $X_1 \subset U_1$ and $X_2 \subset U_2$ with $|X_i| \ge \big( 1 - \gamma \big) |U_i|$ for each $i \in \{1,2\}$. 
\end{defn}

Given $m \in \N$, $\gamma > 0$ and a collection~$\U$ of disjoint sets of vertices in a two-coloured complete graph $G$, we say that $\U$ is \emph{$(m,\gamma)$-connected in $G_R$} if the graph on vertex set~$\U$ whose edges are the $(m,\gamma)$-good pairs in ${\U \choose 2}$ is connected. 

\begin{lemma}\label{lem:tree:embed}
Given $\eps > 0$ and $s,k \in \N$, there exists $n_0 = n_0(\eps,s,k)$ such that the following holds whenever $n \ge n_0$. Let $p = \log_{(k+4)} (n)$, $q = \log_{(k)} (n)$, $2^{8p} \ll m \ll 2^{2^p}$ and $\gamma \ge m^{-1/4}$. 

Let $G$ be a two-coloured complete graph with no blue~$K_s$, and let~$\U$ be a collection of disjoint sets of vertices of $G$ satisfying the following conditions:
\begin{itemize}
 \item[$(a)$] $\sum_{U \in \U} |U| \ge (1+3\eps)2^n$. \smallskip 
 \item[$(b)$] $|U| \ge 2^{n-p}$ for every $U \in \U$.\smallskip  
 \item[$(c)$] $|\Delta_B(U)| \le 2^{n-q}$ for every $U \in \U$. \smallskip
 \item[$(d)$] $\U$ is $(m,\gamma)$-connected in $G_R$.
\end{itemize}
Then $Q_n \subset G_R$.
\end{lemma}

Lemma~\ref{lem:tree:embed} will follow easily from Lemma~\ref{lem:path:embed} once we have refined the partition $\U$ so as to produce an $m$-path $\V$. The following straightforward lemma performs this refinement.

\begin{lemma}\label{lemma:tree_to_path}
Let $G$ be a two-coloured complete graph, let $m \in \N$ and $\gamma > 0$, and let~$\U$ be a collection of disjoint sets of vertices which is $(m,\gamma)$-connected in $G_R$. If 
\begin{equation}\label{eq:treetopath}
\gamma \sqrt{m} \cdot |U| \, \gg \, 2^n \cdot |\U|
\end{equation}
for every $U \in \U$, then there exists a refinement $\V$ of $\U$, satisfying 
$$|V| \ge \frac{|U|}{2|\U|}$$
for each $V \in \V$ with $V \subset U \in \U$, such that $\V$ is an $m$-path in $G_R$. 
\end{lemma}

\begin{proof}
Set $u = |\U| - 1$, and let $(W_0,\ldots,W_{2u})$ be a closed walk on vertex set $\U$ that visits every $U \in \U$ at least once, and such that $(W_i,W_{i+1})$ is $(m,\gamma)$-good for each $0 \le i < 2u$. 
We claim that there exists a collection $\K$ of disjoint red copies of $K_{t,t}$, where $t = \binom{m}{m/2}$, such that each graph $G_R[W_i,W_{i+1}]$ contains exactly $(1 + \eps) 2^{n-m}$ elements of $\K$. Indeed, by Definition~\ref{defn:m_good} we can simply find the elements of $\K$ greedily, noting that, by~\eqref{eq:treetopath}, we use up at most
$$(1 + \eps) 2^{n-m} \cdot t \cdot |\U| \, \le \, \gamma |U|$$
elements of each set $U \in \U$ in the process. 

Now, for each $0 \le i \le 2u$, let $V_i' \subset W_i$ and $V_i'' \subset W_i$ denote the vertices in $W_i$ of the elements of $\K$ in $G_R[W_{i-1},W_i]$ and $G_R[W_i,W_{i+1}]$ respectively (noting that $V_0' = V_{2u}'' = \emptyset$), and let $\V = \{V_0,\ldots,V_{2u}\}$ be an arbitrary refinement of $\U$ such that:
\begin{itemize}
\item[$(a)$] $V_i' \cup V_i'' \subset V_i \subset W_i$ for each $0 \le i \le 2u$. 
\item[$(b)$] If $W_i = W_j$, then $|V_i| \in |V_j| \pm 1$.
\end{itemize}
Since each set $U \in \U$ appears at most $|\U| = u + 1$ times in the multi-set $\{W_1,\ldots,W_{2u}\}$, it follows that $|V| \ge \lfloor |U| / (u+1) \rfloor \ge |U| / 2|\U|$ for every $V \in \V$ with $V \subset U \in \U$. Moreover, since the graph $G_R[V_i,V_{i+1}]$ contains at least $(1 + \eps) 2^{n-m}$ vertex-disjoint copies of $K_{t,t}$, it follows that $\V$ is an $m$-path in $G_R$, as claimed.
\end{proof}

We can now easily deduce Lemma~\ref{lem:tree:embed}.

\begin{proof}[Proof of Lemma~\ref{lem:tree:embed}]
Let $\U$ be a family of disjoint vertex sets as described in the lemma, and note that (by considering a sub-tree if necessary) we may assume that $|\U| \le 2^{p+1}$. We have
$$\gamma \sqrt{m} \cdot |U| \, \gg \, 2^{n+p+1} \, \ge \, 2^n \cdot |\U|$$
for every $U \in \U$, and thus, by Lemma~\ref{lemma:tree_to_path}, there exists a refinement $\V$ of $\U$, satisfying 
$$|V| \ge \frac{|U|}{2|\U|} \, \ge \, 2^{n-3p}$$
for each $V \in \V$, such that $\V$ is an $m$-path in $G_R$. By Lemma~\ref{lem:path:embed}, it follows immediately that $Q_n \subset G_R$, as required.
\end{proof}

We are almost ready to deduce Proposition~\ref{prop:matching_embed} from Lemma~\ref{lem:tree:embed}. The final tool we shall need is the following theorem of K\"ov\'ari, S\'os and Tur\'an~\cite{KST}, which implies that if a pair $\{U,V\}$ is not $m$-good, then the red bipartite subgraph $G_R[U,V]$ has very few edges.

\begin{thm}[K\"ov\'ari, S\'os and Tur\'an~\cite{KST}] 
\label{thm:KST}
If $G$ is an $N \times N$ bipartite graph that does not contain~$K_{t,t}$ as a subgraph, then
\[
 e(G) \, \le \,  O\big( N^{2-1/t} \big).
\]
\end{thm}

We can now prove Proposition~\ref{prop:matching_embed}.

\begin{proof}[Proof of Proposition~\ref{prop:matching_embed}]
Let $C,\eps > 0$ and $s,k \in \N$ be arbitrary, and let $n,p,q \in \N$ be as described in the statement of the lemma. Let $G$ be a two-coloured complete graph on at most $C \cdot 2^n$ vertices with no blue~$K_s$ and no red $Q_n$, and let~$\U$ be a collection of disjoint sets of vertices such that~$(a)$ and~$(b)$ hold. 
Set $m = 2^{9p}$ and $\gamma = 2^{-2p}$, and let
$$\U \, = \, \U_1 \cup \cdots \cup \U_r$$
be a partition of $\U$ into $(m,\gamma)$-components.\footnote{That is, let $\{\U_1,\ldots,\U_r\}$ be the collection of maximal its $(m,\gamma)$-connected sets.}

Suppose first that there exists $j \in [r]$ such that $\sum_{U \in \U_j} |U| \ge (1+3\eps)2^n$. Since $\U_j$ is $(m,\gamma)$-connected in $G_R$, it follows by Lemma~\ref{lem:tree:embed} that $Q_n \subset G_R$, which contradicts our choice of $G$. Thus we have $\sum_{U \in \U_j} |U| \le (1+3\eps)2^n$ for each $j \in [r]$.

We claim that there exists a set $X$ satisfying conditions~$(i)$ and~$(ii)$ of the proposition. In order to prove this, we shall define, for each pair of sets $U,W \in \U$ which are not in the same $(m,\gamma)$-component, sets $X_W(U) \subset U$ and $X_U(W) \subset W$ such that the graph 
$$G[U \setminus X_W(U),W \setminus  X_U(W)]$$ 
is dense in blue. To do so, observe first that the pair $(U,W)$ is not $(m,\gamma)$-good, since otherwise $U$ and $W$ would be in the same $(m,\gamma)$-component. By Definition~\ref{defn:m_good}, it follows that there exist sets $Y_W(U) \subset U$ and $Y_U(W) \subset W$, with 
\begin{equation}\label{eq:YWUbounds}
|Y_W(U)| \ge \big( 1 - \gamma \big) |U| \qquad \text{and} \qquad |Y_U(W)| \ge \big( 1 - \gamma \big) |W|,
\end{equation}
 such that $G_R[Y_W(U),Y_U(W)]$ does not contain a copy of $K_{t,t}$. By the K\"ov\'ari-S\'os-Tur\'an bound (Theorem~\ref{thm:KST}), and since $|U| \ge 2^{n-p} \gg 2^{n/2}$ for every $U \in \U$, it follows that 
$$d_R\big( Y_W(U),Y_U(W) \big) \, \ll \, 2^{-n/2t} \, \ll \, 2^{-n/2^{m+1}} \, \ll \, \frac{1}{n^4}.$$
Set $X_W(U) = U \setminus Y_W(U)$ for each such pair $U$ and $W$, and for each $i \in [r]$ and $U \in \U_i$, define
$$X(U) \, = \, \bigcup_{W \in \U \setminus \U_i} X_W(U).$$
Noting that $|\U| \le C \cdot 2^p$ (by~$(a)$ and our bound on $v(G)$), and using~\eqref{eq:YWUbounds}, it follows that
\begin{equation}\label{eq:boundingX}
|X(U)| \, \le \, \gamma |U| \cdot |\U| \, \le \, \gamma |U| \cdot C \cdot 2^p  \, \ll \, |U|,
\end{equation}
for each $U \in \U$, and hence $X := \bigcup_{U \in \U} X(U)$ satisfies conditions~$(i)$ and~$(ii)$, as required.
\end{proof}

\section{The proof of Theorem~\ref{thm:main}}\label{ProofSec}

Combining the results of Sections~\ref{SimonSec} and~\ref{MatchingSec}, it is now easy to deduce Theorem~\ref{thm:main}. We begin by proving a stability theorem. 

\begin{thm}\label{thm:stability}
For every $\eps > 0$ and $s \in \N$, there exists $\delta = \delta(s) > 0$ and $n_0 = n_0(\eps,s) \in \N$ such that the following holds for every $n \ge n_0$. 

Let $G$ be a two-coloured complete graph on $N \ge (1 - \delta) (s - 1) 2^n$ vertices with no blue $K_s$. Then either there exists a partition of $V(G) = S_0 \cup S_1 \cup \cdots \cup S_{s-1}$ such that:
\begin{itemize}
\item[$(a)$] $|S_0| \le \eps 2^n$ and $|S_j| \le \big( 1 + \eps \big) 2^n$ for every $j \in [s-1]$, \smallskip
\item[$(b)$] $G_R[S_j]$ is a red clique for every $1 \le j \le s-1$,\smallskip
\item[$(c)$] $d_R(S_i,S_j) \le 1/n^2$ for every $1 \le i < j \le s-1$, \smallskip
\item[$(d)$] $|N_B(v) \cap S_j| \le |S_j|/n^2$ for every $v \in S_0$ and some $j = j(v) \in [r]$, 
\end{itemize}
or $Q_n \subset G_R$. 
\end{thm}

\begin{proof}
Set $\delta = 1/2s$ and assume, without loss of generality, that $\eps > 0$ is sufficiently small. Note that we may assume also that $N \le \big( s - 1 + \eps s \big) 2^n$, since if $N$ is at least this large then we will prove that $Q_n \subset G_R$. Let $K = K(\eps^2,s)$ be the constant given by Proposition~\ref{prop:simon}. 

We first claim that there exists $k \le 4K - 3$ and a collection~$\U$ of disjoint vertex sets of $G$ such that, writing $p = \log_{(k+4)}(n)$ and $q = \log_{(k)}(n)$, we have
\begin{itemize}
 \item[$(i)$]  $|\bigcup_{U \in \U} U| \ge \big( 1 - \eps^2 \big) v(G)$. \smallskip
 \item[$(ii)$]  $|U| \ge 2^{n-p}$ for every $U \in \U$. \smallskip
 \item[$(iii)$]  $\Delta(G_B[U]) \le 2^{n-q}$ for every $U \in \U$.
\end{itemize}
To see this, set $a(0) = N$ and 
$$a(K - i) \, = \, \frac{2^n}{\log_{(4i)}(n)}$$
for each $1 \le i \le K - 1$, and note that $a(i+1) \le (\eps/8) a(i)$ for every $0 \le i \le K - 2$, since~$n \ge n_0$. By Proposition~\ref{prop:simon}, it follows that there exists $i \in [K - 1]$ and a collection~$\U$ of disjoint vertex sets of $G$ satisfying~$(i)$,~$(ii)$ and~$(iii)$, with $k = 4i + 1$, as claimed.

We next apply Proposition~\ref{prop:matching_embed} to the collection $\U$. Assuming that $Q_n \not\subset G_R$, it follows that there exists a partition $\U = \U_1 \cup \cdots \cup \U_r$ and a set $X \subset V(G)$ such that 
\begin{itemize}
 \item[$(iv)$] $|X| \le \eps 2^{n-1}$ and $\sum_{U \in \U_i} |U| \le \big( 1 + \eps \big) 2^n$ for every $i \in [r]$.\smallskip
 \item[$(v)$] $d_B\big( U \setminus X,W \setminus X \big) \le 1/n^4$ for every $U \in \U_i$ and $W \in \U_j$ with $i,j \in [r]$ and $i \neq j$.
\end{itemize}
By~$(i)$ and~$(iv)$, and the pigeonhole principle, it follows that $r \ge s-1$. We claim that in fact $r = s - 1$, and that each set $\bigcup_{U \in \U_i} U$ contains a very large clique; both facts will follow by essentially the same argument. In brief, if either property fails to hold then we shall remove vertices of high red degree and use the greedy algorithm to find a blue copy of $K_s$. 

Indeed, for each pair of disjoint vertex sets $S,T \subset V(G)$, let
$$Y_T(S) \, = \, \big\{ v \in S \,:\, |N_R(v) \cap T| \ge |T| / n^2 \big\}$$ 
denote the `vertices of high red degree' in $S$ with respect to $T$ and, for each $U \in \U_i$, let
$$Y(U) \, = \, \bigcup_{W \in\, \U \setminus \U_i} Y_{W \setminus X} \big( U \setminus X \big)$$
denote the vertices of $U \setminus X$ which send many red edges to some set in a different part. We claim that 
\begin{equation}\label{eq:YlessU}
|Y(U)| \, \le \, \frac{|U|}{n}
\end{equation}
for every $U \in \U$. Indeed, it follows from property~$(ii)$ and our assumption that $N = O(s \cdot 2^n)$ that $|\U| = O(s \cdot 2^p) \ll n$, and from property~$(v)$ that $|Y_{W \setminus X}(U \setminus X)| \le |U| / n^2$ for every $U \in \U_i$ and $W \in \U_j$ with $i,j \in [r]$ and $i \neq j$. 


Now, suppose that $r \ge s$. Let $U_1 \in \U_1, \ldots, U_s \in \U_s$, and choose vertices
\begin{equation}\label{eq:choosing:clique}
v_j \, \in \, U_j \setminus \Big( X \cup Y(U_j) \cup N_R(v_1) \cup \cdots \cup N_R(v_{j-1}) \Big)
\end{equation}
for each $j \in [s]$. To see that this is possible, recall that $|N_R(v_i) \cap U_j| \le |U_j| / n^2$ for every $i < j$, which holds since $v_i \not\in Y(U_i)$, and by~\eqref{eq:YlessU}, and note that moreover $|X \cap U| \ll |U|$ for each $U \in \U$, by~\eqref{eq:boundingX}. (Alternatively, we can simply throw out (into $S_0$) all sets $U$ which have a too-large intersection with $X$.) Clearly the vertices $\{v_1,\ldots,v_s\}$ span a blue copy of $K_s$ in $G$, which contradicts our assumption that $G_B$ is $K_s$-free.

Hence we may assume that $r = s - 1$. Set
\begin{equation}\label{eq:defSj}
S_j \, := \, \bigcup_{U \in \U_j} \Big( U \setminus \big( X \cup Y(U) \big) \Big)
\end{equation}
for each $j \in [s-1]$, and suppose that there is a blue edge $\{v_1,v_2\}$ in the set $S_1$, say. Then, by exactly the same argument as before, we can choose vertices
\begin{equation}\label{eq:choosing:clique:vtxs}
v_{j+1} \, \in \, S_j \setminus \Big( N_R(v_1) \cup \cdots \cup N_R(v_j) \Big)
\end{equation}
for $2 \le j \le s - 1$, and once again the vertices $\{v_1,\ldots,v_s\}$ span a blue copy of $K_s$ in $G$. Hence property~$(b)$ holds, and~$(c)$ follows easily from~$(v)$, together with~\eqref{eq:YlessU} and~\eqref{eq:defSj}. 

\enlargethispage{\baselineskip}

Finally, let us set
\begin{equation}\label{eq:defSzero}
S_0 \, = \, V(G) \setminus \big( S_1 \cup \cdots \cup S_{s-1} \big),
\end{equation}
and check that the conditions~$(a)$ and~$(d)$ hold. Indeed, by~\eqref{eq:YlessU} and since $|\U| = O(s \cdot 2^p) \ll n$, and using properties~$(i)$ and~$(iv)$, we have
$$|S_0| \, \le \, \eps^2 v(G) + |X| + \sum_{U \in \U} |Y(U)| \, \le \, \eps 2^n$$
and $|S_j| \le ( 1 + \eps ) 2^n$ for every $j \in [s-1]$. Moreover, if there exists a vertex $v \in S_0$ with $|N_B(v) \cap S_j| \ge |S_j|/n^2$ for every $j \in [r]$, then we can choose vertices $v_2 \in S_1, \ldots,v_s \in S_{s-1}$, as in~\eqref{eq:choosing:clique:vtxs}, so that $\{v,v_2,\ldots,v_s\}$ spans a blue $K_s$, which is a contradiction. This completes the proof of the stability theorem.
\end{proof}

In order to deduce Theorem~\ref{thm:main} from Theorem~\ref{thm:stability}, we shall need the following easy lemma.

\begin{lemma}\label{lem:finalembed}
Let $G$ be a two-coloured complete graph, and let $X,Y \subset V(G)$ be disjoint vertex sets with $|X| + |Y| \ge 2^n$ and $|Y| \le 2^{n-3}$. Suppose that $G_R[X]$ is a clique, and that 
$$| N_B(v) \cap X | \, \le \, |X| / n^2$$
for every $v \in Y$. Then $Q_n \subset G_R$. 
\end{lemma}

\begin{proof}
We construct an embedding $\varphi$ of $Q_n$ into $G_R$ by embedding the lowermost even layers of~$Q_n$ into~$Y$, and then embedding the remaining vertices into~$X$, as follows:
\begin{itemize}
 \item[1.] Let $0 \le \ell < n/4$ be such that
  \[
  \sum_{i=0}^{\ell-1} \binom{n}{2i} \le |Y| < \sum_{i=0}^{\ell} \binom{n}{2i},
  \]
and embed the first $\ell$ even layers of the hypercube arbitrarily into $Y$. If some vertices of $Y$ remain unused, embed part of the $(\ell + 1)$st even layer into them. \smallskip

 \item[2.] For each $\vx \in Q_n$ with $\varphi(\vx) \not\in Y$ and $\|\vx\| \le 2\ell + 1 \le n/2$, select an unused vertex of $X$ which avoids the blue neighbourhoods (in $G$) of its neighbours in $Q_n$. That is,
 $$\varphi(\vx) \, \in \, X \setminus \bigcup_{\vy \sim \vx} N_B\big( \varphi(\vy) \big),$$
where~$\vy$ ranges over the $Q_n$-neighbours of~$\vx$ which have already been embedded. Since $\vx$ has at most $n$ such neighbours, and each has at most $|X| / n^2$ blue neighbours in $X$, it follows that we will never run out of vertices. \smallskip
  
  \item[3.] Embed the remaining elements of~$Q_n$ into the remaining elements of~$X$ arbitrarily.
\end{itemize}
It is easy to see that this gives an embedding of $Q_n$ into $G_R$, as required.
\end{proof}

We are finally ready to determine the Ramsey number $r(K_s,Q_n)$. 

\begin{proof}[Proof of Theorem~\ref{thm:main}]
Let~$G$ be a two-coloured complete graph on $(s-1)(2^n-1)+1$ vertices that does not contain a blue~$K_s$; we claim that $Q_n \subset G_R$. By Theorem~\ref{thm:stability}, either $Q_n \subset G_R$ or there exists a partition of $V(G) = S_0 \cup S_1 \cup \cdots \cup S_{s-1}$ satisfying properties $(a)$-$(d)$ of that theorem. By property~$(d)$, we may assign each vertex $v \in S_0$ to a set $S_j$ such that $|N_B(v) \cap S_j| \le |S_j|/n^2$. We thus obtain a partition
$$V(G) = T_1 \cup \cdots \cup T_{s-1}$$
such that $T_j$ is the union of a red clique $S_j$ and a set with low blue degree into that clique. By the pigeonhole principle, we must have $|T_j| \ge 2^n$ for some $j \in [s-1]$. By Lemma~\ref{lem:finalembed}, it follows that $Q_n \subset G_R[T_j]$, as required.
\end{proof}

We end the paper by noting that our proof can easily be modified to prove the following generalization of Theorem~\ref{thm:main}. Given a graph $H$ with chromatic number $s$, let $\sigma(H)$ denote the size of the smallest colour class in any $s$-colouring of $H$.


\begin{thm}\label{thm:generalH}
Let $H$ be a graph. Then
\begin{equation}\label{eq:genHthm}
 r(H,Q_n) = \big( \chi(H) - 1 \big)\big( 2^n - 1 \big) + \sigma(H)
\end{equation}
for every $n \ge n_0(H)$.
\end{thm}

Note that the lower bound in~\eqref{eq:genHthm} is trivial, since a blue complete $s$-partite graph with $s-1$ parts of size $2^n - 1$ and one part of size $\sigma(H) - 1$ contains no blue copy of $H$, and no red copy of $Q_n$. In order to prove Theorem~\ref{thm:generalH} we have two tasks: we must prove a version of Theorem~\ref{thm:stability} for blue $H$-free colourings, and we must prove a corresponding embedding lemma. We begin by stating the generalized stability theorem.

\begin{thm}\label{thm:Hstability}
For every $\eps > 0$ and every graph $H$ with chromatic number $s$, there exist $C > 2$, $\delta > 0$ and $n_0 = n_0(C,\eps,H) \in \N$ such that the following holds for every $n \ge n_0$.

Let $G$ be a two-coloured complete graph on $N \ge (1 - \delta) (s - 1) 2^n$ vertices with no blue copy of~$H$. Then either there exists a partition of $V(G) = S_0 \cup S_1 \cup \cdots \cup S_{s-1}$ such that:
\begin{itemize}
\item[$(a)$] $|S_0| \le \eps 2^n$ and $|S_j| \le \big( 1 + \eps \big) 2^n$ for every $j \in [s-1]$, \smallskip
\item[$(b)$] $d_B(S_j) \le 1 / n^C$ for every $1 \le j \le s-1$,\smallskip
\item[$(c)$] $d_R(S_i,S_j) \le 1/n^C$ for every $1 \le i < j \le s-1$, \smallskip
\item[$(d)$] $|N_B(v) \cap S_j| \ge |S_j|/n^2$ for every $j \in [r]$ for at most $n^C$ vertices $v \in S_0$, 
\end{itemize}
or $Q_n \subset G_R$. 
\end{thm}

The proof of Theorem~\ref{thm:Hstability} is very similar to that of Theorem~\ref{thm:stability}. Indeed, since an $H$-free graph is obviously $K_{v(H)}$-free, we may still apply Propositions~\ref{prop:simon} and~\ref{prop:matching_embed} to the blue graph and the collection $\U$, respectively. Moreover, using the greedy algorithm exactly as in~\eqref{eq:choosing:clique}, if $r \ge s$ then one easily obtains a copy of $H$, simply by choosing $v(H)$ vertices at each step instead of only one. Similarly, if $S_1$ contains a blue complete bipartite graph with $v(H)$ vertices in each part, then we can again find a blue copy of $H$, as in~\eqref{eq:choosing:clique:vtxs}. It follows, by Theorem~\ref{thm:KST}, that $d_B(S_j) \le 1 / n^C$ for every $1 \le j \le s-1$, as required.

We thus obtain properties~$(a)$,~$(b)$ and~$(c)$, as before. In order to prove property~$(d)$, we need the following claim, which follows via a simple application of dependent random choice (see~\cite{DRC} for an excellent survey of this technique). Let $G$ be as in the statement of Theorem~\ref{thm:Hstability}, let $S_0,\ldots,S_{s-1}$ be the sets defined as in~\eqref{eq:defSj} and~\eqref{eq:defSzero}, and set
$$A \, := \, \big\{ v \in S_0 : |N_B(v) \cap S_j| \ge 2^n / n^2 \textup{ for every } j \in [s-1] \big\}.$$
The claim says that either $A$ is small, or there exists a subset of $A$ of size $\sigma(H)$ with sufficiently large common blue neighbourhood in each set $S_j$ that we can find a blue copy of $H$ using the greedy algorithm. Let us assume, as we may, that $C$ is sufficiently large.

\begin{claim*}
Either $|A| \le n^C$, or there exists a set $X \subset A$ with $|X| = \sigma(H)$ such that 
\begin{equation}\label{eq:claim:bignbhd}
\bigg| \bigcap_{v \in X} N_B(v) \cap S_j \bigg| \, \ge \, \frac{2^n}{n^{C-1}}
\end{equation}
for every $j \in [s-1]$. 
\end{claim*}

\begin{proof}
With foresight, set $t = \lfloor (\log |A|) / (2s \log n) \rfloor$, and choose $t$ elements of each set $S_j$ uniformly at random, with repetition. Let $X$ be the common blue  neighbourhood of these vertices in $A$. We claim that $X$ has the required properties with positive probability.

To see this, let us first calculate the expected size of $X$. By our choice of $t$, it is
$$\sum_{v \in A} \prod_{j = 1}^{s-1} \left( \frac{|N_B(v) \cap S_j|}{|S_j|} \right)^t \, \ge \, |A| \cdot n^{-2(s-1)t} \, \gg \, \sigma(H),$$
assuming that $|A| \ge n^C$. Next, let us calculate the expected number of subsets of $X$ of size $\sigma(H)$ which have at most $2^n / n^{C-1}$ common neighbours in some $S_j$. It is at most
$${|A| \choose \sigma(H)} \left( \frac{1}{n^{C-1}} \right)^t \, \le \, |A|^{\sigma(H)} n^{-Ct/2} \, \ll \, 1$$
if $|A| \ge n^C$, again by our choice of $t$. It follows immediately that $X$ has the desired properties with positive probability, and hence that there exists such a set $X \subset A$, as required. 
\end{proof}

It is easy to see that, using the greedy algorithm exactly as before, if there is a set $X \subset A$ as in the claim, then $H \subset G_B$. Hence we may assume that $|A| \le n^C$, as claimed. This completes the proof of Theorem~\ref{thm:Hstability}. 

In order to deduce Theorem~\ref{thm:generalH}, first let $T_0 \subset A$ denote the set of vertices of $S_0$ which have red degree at most $\eps 2^n$ into every set $S_j$. Note that if $|T_0| \ge \sigma(H)$ then these vertices form a set $X$ satisfying~\eqref{eq:claim:bignbhd} (assuming $\eps$ was chosen sufficiently small), in which case $H \subset G_B$, as noted above. Now, by the definition of $A$, we may partition the vertices of $S_0 \setminus A$ according to the set $S_j$ into which they send at most $2^n / n^2$ blue edges, and the vertices of $A \setminus T_0$ according to the set $S_j$ into which they send at least $\eps 2^n$ red edges. We thus obtain a partition
$$V(G) = T_0 \cup T_1 \cup \cdots \cup T_{s-1}$$
where $|T_0| \le \sigma(H) - 1$, and for each $j$ we have $T_j = S_j \cup Z_j \cup A_j$, where $Z_j \subset S_0 \setminus A$ and $A_j \subset A \setminus T_0$ are the parts of the partitions described above corresponding to $S_j$. 

By the pigeonhole principle, some set $T_j$ must contain at least $2^n$ elements; it only remains to show that $Q_n \subset G_R[T_j]$. We embed greedily, first placing the (at most $n^C$) vertices of $A_j$ into even layers at pairwise distance at least three from one another, then embedding the vertices of $Z_j$, and finally those of $S_j$ (embedding neighbours of $\varphi(A_j)$ first, then neighbours of $\varphi(Z_j)$, then the rest), as in Lemma~\ref{lem:finalembed}. The fact that $S_j$ is no longer a red clique (but rather a very dense red set) can be dealt with by a standard `vertex-switching' argument, first introduced in~\cite{SS}, and used (for example) in~\cite[Lemma~21]{ABS}. Briefly, given any embedding of $Q_n$ into $V(G)$, and any blue edge of $G$ corresponding to an edge of $Q_n$, we can find two vertices of $G$ such that, if we swap their pre-images, then we obtain an embedding with fewer blue edges. This completes the proof of Theorem~\ref{thm:generalH}.

\vskip-0.5cm

\end{document}